\documentclass[12pt]{article}
\usepackage{amsmath,amsfonts,amsthm,amssymb, mathtools}
\usepackage{mathrsfs, graphicx,color,latexsym, tikz, calc,cite,enumerate,indentfirst}
\usepackage{booktabs}
\usepackage{tabularx}
\usetikzlibrary{shadows}
\usetikzlibrary{patterns,arrows,decorations.pathreplacing}
\usepackage[colorlinks=true,
linkcolor=blue,citecolor=blue,
urlcolor=blue]{hyperref}

\newtheorem{theorem}{Theorem}[section]
\newtheorem{lemma}{Lemma}[section]

\newtheorem{corollary}{Corollary}[section]

\newtheorem{example}{Example}[section]

\allowdisplaybreaks

\title{\bf \Large Fractional revival on quasi-abelian Cayley graphs}

\author{Yi Fang$^a$, Xueyi Huang$^{a,}$\footnote{Corresponding author.}, \setcounter{footnote}{-1}\footnote{\emph{Email address:} fangyimath@163.com (Y. Fang),  huangxy@ecust.edu.cn (X. Huang), xiaogliu@nwpu.edu.cn (X. Liu),  zhanxfmath@163.com (X. Zhan).} Xiaogang Liu$^{b,c,d}$, Xiongfeng Zhan$^a$\\[2mm]
	\small $^a$School of Mathematics, East China University of Science and Technology, \\
	\small  Shanghai 200237, P. R. China\\
	\small $^b$School of Mathematics and Statistics,
Northwestern Polytechnical University,\\
\small Xi'an, Shaanxi 710072, P.R. China\\
\small $^c$Research \& Development Institute of Northwestern Polytechnical University in Shenzhen,\\
\small  Shenzhen, Guangdong 518063, P.R. China\\
\small $^d$Xi'an-Budapest Joint Research Center for Combinatorics,
Northwestern Polytechnical University,\\ \small  Xi'an, Shaanxi 710129, P.R. China
}

\date{ }

\begin{document}

\maketitle

\begin{abstract}

Fractional revival, a quantum transport phenomenon critical to entanglement generation in quantum spin networks, generalizes the notion of perfect state transfer on graphs. A Cayley graph 
$\mathrm{Cay}(G,S)$ is called quasi-abelian if its connection set  $S$ is a union of conjugacy classes of the group $G$. In this paper, we establish a necessary and sufficient condition for quasi-abelian Cayley graphs to have fractional revival. This extends a result of Cao and Luo (2022) on the existence of fractional revival in Cayley graphs over abelian groups.

	\par\vspace{2mm}
		
    
\end{abstract}

\baselineskip=0.202in

\section{Introduction}

The development of reliable quantum state transfer methods is of fundamental importance in quantum information theory. Bose \cite{B03} introduced a protocol employing spin chains as short-range quantum communication channels to achieve high-fidelity quantum state transmission. A qubit is said to exhibit perfect state transfer within a spin chain if its transmission fidelity reaches unity $1$. In this framework, spin chains are modeled as graphs, where vertices correspond to qubits and edges represent quantum couplings. The temporal evolution of quantum states can be simulated through the adjacency matrices associated with these graphs.

Let $\Gamma$ be a connected graph with vertex set $V(\Gamma)$ and edge set $E(\Gamma)$. The \textit{adjacency matrix} of $\Gamma$ is defined as $A: = A(\Gamma) = (a_{u,v})_{u,v \in V(\Gamma)}$, where $a_{u,v} = 1$ if $uv \in E(\Gamma)$, and $a_{u,v} = 0$ otherwise. The \textit{transition matrix} of $\Gamma$ with respect to $A$ is given by 
\begin{equation*}
    H(t) = \exp(\mathrm{i} tA) = \sum_{k=0}^\infty \frac{\mathrm{i}^k A^k t^k}{k!},~t \in \mathbb{R},~\mathrm{i} = \sqrt{-1}.
\end{equation*}
Let $\mathbb{C}^{|\Gamma|}$ ($|\Gamma|=|V(\Gamma)|$) denote the $n$-dimensional vector space over the complex field, and let $\mathbf{e}_u$ be the standard basis vector in $\mathbb{C}^n$ indexed by the vertex $u$. If $u$ and $v$ are distinct vertices in $\Gamma$ and there exists a time $t$ such that
\begin{equation*}
   H(t)\mathbf{e}_u=\gamma\mathbf{e}_v, 
\end{equation*} 
where $\gamma$ is a complex number and $|\gamma|=1$, then we say $\Gamma$ admits \textit{perfect state transfer}  from $u$ to $v$ at time $t$. 

The study of perfect state transfer originated in the work of Bose \cite{B03} and Christandl et al. \cite{CDDE05, CDDE04}. In recent years, the problem of characterizing those graphs that admits perfect state transfer has attracted increasing attention from researchers. In 2012, Godsil \cite{G12} provided a survey on perfect state transfer from a mathematical perspective.  For more results on perfect state transfer, we refer the reader to \cite{BP09, CCL20, WC22, TFC19, LCW22, CG11, DC16, CGGC15, CL15, G12, G10, LLZ21, WAF24, ZLZ20, PB17} and  references therein.

As a generalization of perfect state transfer, fractional revival was introduced by Pohill and Sudheesh \cite{RS15}. For distinct vertices  $u$ and $v$ in $\Gamma$, if there exists a time $t$ such that
\begin{equation*}
    H(t)\mathbf{e}_u = \alpha\mathbf{e}_u + \beta\mathbf{e}_v,
\end{equation*}
where $\alpha$ and $\beta$ are complex numbers satisfying $\beta\neq 0$ and $|\alpha|^2 + |\beta|^2 = 1$, we say that $\Gamma$ admits \textit{fractional revival} from $u$ to $v$ at time $t$, or  $(\alpha,\beta)$-\textit{revival}  occurs on $\Gamma$ from $u$ to $v$ at time $t$. In particular, if $\alpha = 0$, then $\Gamma$ has perfect state transfer from $u$ to $v$ at time $t$. Moreover, it was shown in \cite{CCTVZ19} that if $\Gamma$ has fractional revival from $u$ to $v$ at time $t$, then $\Gamma$ has fractional revival from $v$ to $u$ at the same time. Thus, we simply
say $\Gamma$ has fractional revival between $u$ and $v$ at time $t$. 

Up to now, there are  few results on fractional revival. Genest, Vinet and Zhedanov \cite{GVZ16} systematically studied quantum spin chains with fractional revival at two sites. Luo, Cao, Wang and Wu  \cite{LCW22} showed that cycles admit fractional revival  only when they have four or six vertices, while paths admit fractional revival only when they have two, three or four vertices. Godsil and Zhang \cite{GZ22} proposed a class of graphs admitting fractional revival between non-cospectral vertices. Chan, Coutinho,  Tamon, Vinet and Zhan \cite{CCTVZ20} investigated fractional revival in graphs whose adjacency matrices belong to the Bose–Mesner algebra of association schemes. Wang, Wang and Liu \cite{WWL24} established a necessary and sufficient condition for Cayley graphs
over finite abelian groups $G$ to have fractional revival between vertices $x$ and $x + a$, where $a$ is a fixed  element $G$. 
Cao and Luo \cite{CL22} characterized fractional revival between any two distinct vertices in Cayley graphs over abelian groups. Additionally, Wang, Wang and Liu \cite{WWL23} derived criteria for fractional revival in semi-Cayley graphs over finite abelian groups.

Let $G$ be a finite group with identity element $1$, and let $S$ be an inverse closed subset of $G$ with $1\notin S$.  The \textit{Cayley graph} of $G$ with respect to $S$, denoted by  $\mathrm{Cay}(G,S)$, is defined as the graph with vertex set $G$ and edge set $\{\{g,sg\}\mid g\in G, s\in S\}$. We say that  $\mathrm{Cay}(G,S)$ is  \textit{quasi-abelian} (or \textit{normal}) if $S$ is the union of some conjugacy classes of $G$ (cf. \cite{WX97}). Clearly, all Cayley graphs over abelian groups are quasi-abelian.

This paper investigates fractional revival in quasi-abelian Cayley graphs. We present a necessary and sufficient condition for the existence of fractional revival in such graphs and prove that all quasi-abelian Cayley graphs admitting fractional revival are integral. This generalizes the main result of Cao and Luo \cite{CL22}. Furthermore, we observe that every connected quasi-abelian Cayley graph over the symmetric group $S_n$ for $n\geq 3$ does not admit fractional revival.

The paper is organized as follows. In Section \ref{sec::2}, we determine the transition matrix structure for connected Cayley graphs with fractional revival, and review some notations and results from the representation theory of finite groups. In Section \ref{sec::3}, we establish the characterization of fractional revival in quasi-abelian Cayley graphs and recover the main result from \cite{CL22} as a special case.

\section{Preliminaries}\label{sec::2}
Let $G$ be a finite group, and let $\Gamma$ be a Cayley graph on $G$. As $\Gamma$ is vertex-transitive, it is easy to see that  if $(\alpha,\beta)$-revival occurs on $\Gamma$ from some vertex at time $t$, then $(\alpha,\beta)$-revival occurs from every vertex of $\Gamma$ at the same time $t$ (see \cite{CCTVZ20}). Based on this fact, we can deduce the following result.
\begin{lemma}\label{lem::1}(See also \cite{CL22})
    Let $G$ be a finite group, and let $\Gamma$ be a connected Cayley graph over $G$. Suppose that $\alpha$ and $\beta$ are two complex numbers satisfying $\beta\neq 0$ and $|\alpha|^2+|\beta|^2=1$. Then $(\alpha,\beta)$-revival occurs on $\Gamma$ from a vertex $u$ to a vertex $v$ at time $t$ if and only if 
    \begin{equation*}
        H(t) = \alpha I + \beta Q,
    \end{equation*}
    where $Q$ is a symmetric permutation matrix with zero diagonals and  $\alpha\overline{\beta} + \overline{\alpha} \beta = 0$.
\end{lemma}

\begin{proof}
 Suppose that $H(t)= \alpha I + \beta Q$ with $Q$ being a symmetric permutation matrix with zero diagonals. Then, for any $u\in G$, there exists some $v\in G\setminus\{u\}$ such that  $Q_{u,v} = 1$ and $Q_{u,w}=0$ for all $w\neq v$. It follows that   $H(t)_{u,u}=\alpha$, $H(t)_{u,v}=\beta$, and $H(t)_{u,w}=0$ for all $w\in G\setminus\{u,v\}$. Therefore,  we conclude that $(\alpha,\beta)$-revival occurs on $\Gamma$ from $u$ to $v$ at time $t$.

    Conversely, if $(\alpha,\beta)$-revival occurs on $\Gamma$ from some vertex $u$ to another vertex $v$, then $H(t) \mathbf{e}_{u}=\alpha \mathbf{e}_{u}+\beta \mathbf{e}_{v}$, that is,  $H(t)_{u,u}=\alpha$, $H(t)_{u,v}=\beta$, and $H(t)_{u,w}=0$ for all $w\in G\setminus\{u,v\}$. Since $\Gamma$ is a Cayley graph on $G$, the adjacency matrix $A=(a_{x,y})_{x,y\in G}$ satisfies the condition that $a_{xg,yg}=a_{x,y}$ for all $g\in G$. Thus, for any $g\in G$, we have $H(t)_{ug,ug}=\alpha$, $H(t)_{ug,vg}=\beta$, and $H(t)_{ug,wg}=0$ whenever $w\in G\setminus\{u,v\}$. By the arbitrariness of $g\in G$, we assert that $H(t)=\alpha I+\beta Q$, where $Q$ is a permutation matrix with zero diagonals. Moreover, since $H(t)$ is a symmetric unitary matrix, we have  $Q^T=Q$ and $H(t)H(t)^*=H(t)\overline{H(t)}=I$, that is, $\alpha\overline{\alpha}I+\beta\overline{\beta}Q^2+(\alpha\overline{\beta}+\overline{\alpha}\beta)Q=I$.   Combining this with the fact that $Q^2=I$ and $|\alpha|^2+|\beta|^2=1$, we obtain 
            $\alpha\overline{\beta}+\overline{\alpha}\beta=0$.
            
    We complete the proof.
\end{proof}

	Let $G$ be a finite group, and let $V$ be a $d$-dimensional vector space over the  complex field $\mathbb{C}$. A \textit{representation} of $G$ on $V$ is a group homomorphism $\rho: G \rightarrow GL(V)$, where $GL(V)$ is the general linear group consisting of invertible linear transformations on $V$. The dimension $d$ is referred to as the \textit{degree} of the representation $\rho$. The \textit{trivial representation} of $G$ is the homomorphism $\rho_0: G \rightarrow \mathbb{C}^*$ defined by letting $\rho_0(g) = 1$ for all $g \in G$, where $\mathbb{C}^*$ is the multiplicative group of $\mathbb{C}$.
	
	A subspace $W$ of $V$ is said to be $G$-\textit{invariant} if for every $g \in G$ and $w \in W$, we have $\rho(g)w \in W$. Clearly, both the zero subspace $\{0\}$ and the entire space $V$ are $G$-invariant subspaces. A representation $\rho: G \rightarrow GL(V)$ is called \textit{irreducible} if $\{0\}$ and $V$ are the only $G$-invariant subspaces of $V$.
	
Let  $\rho$ and $\tau$ be two representations of $G$ on vector spaces $V$ and $W$, respectively. We say that $\rho$ and $\tau$ are \textit{equivalent}, denoted by $\rho \sim \tau$, if there exists a linear isomorphism $f: V \rightarrow W$ such that $\tau(g) = f\rho(g)f^{-1}$ for all $g\in G$.

	The \textit{character} associated with a representation $\rho: G \rightarrow GL(V)$ is a function $\chi_\rho: G \rightarrow \mathbb{C}$ defined by letting $\chi_\rho(g) = \text{tr}(\rho(g))$ for all $g \in G$, where $\mathrm{tr}(\rho(g))$ is the trace of the representation matrix of $\rho(g)$ with respect to some basis of $V$. It is known that two representations of $G$ are equivalent if and only if they share the same character.  The \textit{degree} of the character $\chi_\rho$, denoted by $d_{\chi_\rho}$, is just the degree of the representation $\rho$. Clearly, $d_{\chi_{\rho}}=\chi_{\rho}(1)$. A character $\chi_\rho$ is called \textit{irreducible} if the representation it corresponds to, $\rho$, is itself irreducible. The set of all irreducible characters (resp., unitary representatives of equivalent classes of irreducible representations) of $G$ is denoted by $\widehat{G}$ (resp., $\mathrm{Irr}(G)$). For any $\chi\in \widehat{G}$, we use $\rho^\chi$ to denote the unitary irreducible representation  in $\mathrm{Irr}(G)$  with character  $\chi$.

\begin{lemma}\label{lem::2}
Let $G$ be a finite group, and let $a\in G\setminus\{1\}$. Then $a$ is an element of order $2$ that lies in the center of $G$  if and only if $\chi(a)=\pm d_\chi$ for all $\chi\in \widehat{G}$. In this case, we have 
$$\sum_{\chi\in \widehat{G}_0}d_\chi^2=\sum_{\chi\in \widehat{G}_1}d_\chi^2=\frac{|G|}{2},$$
where $\widehat{G}_0 = \{\chi\in \widehat{G}\mid \chi(a)=d_\chi\}$ and $\widehat{G}_1=\{\chi\in \widehat{G}\mid \chi(a)=-d_\chi\}$.
\end{lemma}
\begin{proof}
First assume that $\chi(a)=\pm d_\chi$ for all $\chi\in \widehat{G}$. Let $m$ be the exponent of $G$, i.e., the least common multiple of the orders of the elements in $G$.  Since $\rho^\chi(a)$ is a unitary matrix of order $d_\chi $ satisfying $(\rho^\chi(a))^m=\rho^\chi(a^m)= I$  and $\mathrm{tr}(\rho^\chi(a))=\chi(a)=\pm d_\chi$, we assert that the eigenvalues of  $\rho^\chi(a)$ are either all $1$ or all $-1$, that is,  $\rho_\chi(a) = \pm I$. Hence,  $\rho_\chi(a^2)=(\rho_\chi(a))^2= I$ for all $\chi\in\widehat{G}$. By the  Fourier inversion formula  (see, for example, \cite[Theorem 5.5.4]{Ste12}), we  immediately deduce that $a^2=1$. Therefore, $a$ is order $2$ and $|G|$ is even.  Let $C_{a}$ denote the conjugacy class of $G$ containing $a$. By the second orthogonality relation for irreducible characters of $G$,  we obtain 
    \begin{align*}
        1 &= \frac{1}{|G|}\sum_{\chi\in\widehat{G}}|C_{a}|\chi(a)\overline{\chi(a)}= \frac{1}{|G|}\sum_{\chi\in\widehat{G}}|C_{a}|d_\chi^2= |C_{a}|,
    \end{align*}
which implies that $a$ lies in the center of $G$. Conversely, assume that $a$ is an element of order $2$ that lies in the center of $G$. For any $\chi\in\widehat{G}$, we have $\rho^\chi(g)\rho^\chi(a)=\rho^\chi(ga)=\rho^\chi(ag)=\rho^\chi(a)\rho^\chi(g)$ whenever $g\in G$. By Schur's lemma, this  implies that $\rho^\chi(a)$ is a scalar matrix. On the other hand, $(\rho^\chi(a))^2=\rho^\chi(a^2)=\rho^\chi(1)=I$. Thus we obtain $\rho^\chi(a)=\pm I$, and it follows that $\chi(a)=\mathrm{tr}(\rho^\chi(a))=\pm d_\chi$, as desired. Moreover, again by the second orthogonality relation for irreducible characters of $G$, we have 
$$
0=\frac{1}{|G|}\sum_{\chi\in \widehat{G}}\chi(a)\overline{\chi(1)}=\frac{1}{|G|}\sum_{\chi\in \widehat{G}}\chi(a)d_\chi=\frac{1}{|G|}\left(\sum_{\chi\in \widehat{G}_0}d_\chi^2-\sum_{\chi\in \widehat{G}_1}d_\chi^2\right),
$$
and hence 
$$\sum_{\chi\in \widehat{G}_0}d_\chi^2 = \sum_{\chi\in \widehat{G}_1}d_\chi^2=\frac{1}{2}\sum_{\chi\in \widehat{G}}d_\chi^2=\frac{|G|}{2}.$$

This completes the proof.
\end{proof}

\section{Main results}\label{sec::3}
Suppose that   $\Gamma=\mathrm{Cay}(G,S)$ is a quasi-abelian Cayley graph. According to \cite{Z88}, the eigenvalues of the adjacency matrix $A$ of $\Gamma$ are given by 
\begin{equation*}
    \lambda_\chi=\frac1{d_\chi}\sum_{g\in S}\chi(g),~\chi\in\widehat{G},
\end{equation*}
and each $\lambda_\chi$ has multiplicity  $d_\chi^2$. Note that $\sum_{\chi\in\widehat{G}}d_{\chi}^2=|G|$. Moreover, the vectors
\begin{equation*}
    \boldsymbol{v}_{ij}^{\chi}=\frac{\sqrt{d_\chi}}{\sqrt{|G|}}\big(\rho_{ij}^{\chi}(g):g\in G \big)^T,~1\leq i,j\leq d_\chi,
\end{equation*}
form an orthonormal basis for the eigenspace with respect to  $\lambda_\chi$, where $\rho_{ij}^{\chi}(g)$  denotes the $(i,j)$-entry of  $\rho^\chi(g)$ (see, for example, \cite{WAF24}).

The spectral decomposition of the adjacency matrix $A$ of $\Gamma$ is 
$$A=\sum_{\chi\in\widehat{G}}\lambda_\chi E_\chi,$$
 where 
\begin{equation*}
    E_\chi= \sum_{i=1}^{d_\chi} \sum_{j=1}^{d_\chi} \boldsymbol{v}_{ij}^{\chi}\left(\boldsymbol{v}_{ij}^{\chi}\right)^*.
\end{equation*}
Then the transition matrix of $\Gamma$ can be written as
\begin{equation*}
    H(t)=\sum_{\chi\in\widehat{G}}\exp(\mathrm{i}t\lambda_\chi)E_\chi.
\end{equation*}
Note that, for any $u,v\in V(\Gamma)=G$, we have
\begin{equation*}
    (E_\chi)_{u,v} = \frac{d_\chi}{|G|} \sum_{i=1}^{d_\chi} \sum_{j=1}^{d_\chi} \rho_{ij}^{\chi}(u) \overline{\rho_{ij}^{\chi}(v)} = \frac{d_\chi}{|G|} \chi(uv^{-1}),
\end{equation*}
and consequently,
\begin{equation}\label{eq::1}
    H(t)_{u,v}= \sum_{\chi\in\widehat{G}}\exp(\mathrm{i}t\lambda_\chi)(E_\chi)_{u,v} = \frac{1}{|G|} \sum_{\chi\in\widehat{G}} d_\chi \chi(uv^{-1}) \exp(\mathrm{i}t\lambda_\chi).
\end{equation}

We are now ready to present the main result of this paper.
\begin{theorem}\label{thm::1}
    Let $\Gamma = \mathrm{Cay}(G,S)$ be a connected quasi-abelian Cayley graph. Suppose that $\alpha$ and $\beta$ are two nonzero complex numbers. Then $(\alpha,\beta)$-revival occurs on $\Gamma$ from a vertex $u$ to a vertex $v$ at time $t$ if and only if the following statements hold:
    \begin{enumerate}[(a)]
    \item $\chi(uv^{-1})=\pm d_\chi$ for all $\chi\in\widehat{G}$, or equivalently,  $|G|$ is even and $uv^{-1}$ is an element of order two that lies in the center of $G$;
    \item $
    \exp(\mathrm{i}t\lambda_\chi)=\alpha+\beta$ if $\chi\in \widehat{G}_0$ and $
    \exp(\mathrm{i}t\lambda_\chi)=\alpha-\beta$ if $\chi\in \widehat{G}_1$, where 
   $\widehat{G}_0 = \{\chi\in \widehat{G}\mid \chi(uv^{-1})=d_\chi\}$ and $\widehat{G}_1=\{\chi\in \widehat{G}\mid \chi(uv^{-1})=-d_\chi\}$.
    \end{enumerate}
   In particular, $\Gamma$ is an integral graph, i.e., $\lambda_\chi\in\mathbb{Z}$ for all  $\chi\in \widehat{G}$.
\end{theorem}

\begin{proof}
 Suppose that $(\alpha,\beta)$-revival occurs on $\Gamma$ from a vertex $u$ to a vertex $v$ at time $t$. Then $|\alpha|^2+|\beta|^2=1$.  According to \eqref{eq::1}, for any $w\in G$,  we have
    \begin{align*}
        H(t)_{u,w} = \frac{1}{|G|} \sum_{\chi\in\widehat{G}} d_\chi \chi(uw^{-1})\exp(\mathrm{i}t\lambda_\chi) =
         \begin{cases}
            \alpha, & \mbox{if $w=u$},\\
            \beta, & \mbox{if $w=v$},\\
            0, &\mbox{otherwise}.
        \end{cases}
    \end{align*}
 Let $f$ be the complex function on $G$ defined by letting 
    \begin{align*}
        f(g)= \frac{1}{|G|}\sum_{\varphi\in\widehat{G}} d_\varphi \varphi(g)\exp(\mathrm{i}t\lambda_\varphi)~\mbox{for}~g\in G. 
    \end{align*}
Then we see that   
    $$
    f(g)=\begin{cases}
            \alpha, &\mbox{if $g=1$},\\
            \beta, &\mbox{if $g=uv^{-1}$},\\
            0, &\mbox{otherwise}.
        \end{cases}
      $$
For any irreducible character $\chi\in\widehat{G}$, 
    \begin{align*}
        \sum_{g\in G} f(g)\overline{\chi(g)}& = \sum_{\varphi \in \widehat{G}}d_\varphi\exp(\mathrm{i}t\lambda_{\varphi}) \cdot \frac{1}{|G|} \sum_{g\in G} \varphi(g)\overline{\chi(g)} \\
        &= \sum_{\chi \in \widehat{G}}d_\varphi\exp(\mathrm{i}t\lambda_{\varphi})\delta_{\varphi,\chi}\\
        &=d_\chi\exp(\mathrm{i}t\lambda_{\chi}),
    \end{align*}
   where $\delta_{\varphi,\chi}=1$ if $\varphi=\chi$, and $\delta_{\varphi,\chi}=0$ otherwise. On the other hand, 
   $$
   \sum_{g\in G} f(g)\overline{\chi(g)}=\alpha \overline{\chi(1)} + \beta \overline{\chi(uv^{-1})}=\alpha d_\chi + \beta \overline{\chi(uv^{-1})}.
   $$
 Thus we have
    \begin{equation}\label{eq::2}
        \exp(\mathrm{i}t\lambda_\chi) = \alpha + \frac{\overline{\chi(uv^{-1})}}{d_\chi}\beta,
    \end{equation}
and it follows that
    \begin{align*}
        H(t)_{u,v} &= \frac{1}{|G|} \sum_{\chi\in\widehat{G}} d_\chi \chi(uv^{-1}) \exp(\mathrm{i}t\lambda_\chi)\\
        & = \frac{1}{|G|} \sum_{\chi\in\widehat{G}}d_\chi \chi(uv^{-1})\left(\alpha+\frac{\overline{\chi(uv^{-1})}}{d_\chi}\beta\right)\\
        &= \frac{1}{|G|}\left(\alpha\sum_{\chi\in\widehat{G}}\chi(1) \chi(uv^{-1})+\beta\sum_{\chi\in\widehat{G}} |\chi(uv^{-1})|^2\right)\\
        &= \frac{\beta}{|G|}\sum_{\chi\in\widehat{G}} |\chi(uv^{-1})|^2.
    \end{align*}
   Combining this with the fact that $H(t)_{u,v}=\beta\neq0$, we obtain
    \begin{equation*}
        \sum_{\chi\in\widehat{G}} |\chi(uv^{-1})|^2 = |G|.
    \end{equation*}
 Note that $|\chi(uv^{-1})|\leq \chi(1)=d_\chi$ for all $\chi\in\widehat{G}$. Then we have
    \begin{equation*}
        |G| = \sum_{\chi\in\widehat{G}}|\chi(uv^{-1})|^2 \leq \sum_{\chi\in\widehat{G}}d_\chi^2=|G|,
    \end{equation*}
and hence  $|\chi(uv^{-1})|=d_\chi$ for all $\chi\in\widehat{G}$. Furthermore, from  \eqref{eq::2} we  can deduce that
    \begin{align*}
        1 &= \left(\alpha+\frac{\overline{\chi(uv^{-1})}}{d_\chi}\beta\right)\overline{\left(\alpha+\frac{\overline{\chi(uv^{-1})}}{d_\chi}\beta\right)}\\
        &= |\alpha|^2 + \frac{|\chi(uv^{-1})|^2}{d_\chi^2}|\beta|^2 + \frac{\overline{\chi(uv^{-1})}}{d_\chi}\overline{\alpha}\beta + \frac{\chi(uv^{-1})}{d_\chi}\alpha\overline{\beta}\\
        &= |\alpha|^2+|\beta|^2+ \frac{\overline{\chi(uv^{-1})}}{d_\chi}\overline{\alpha}\beta + \frac{\chi(uv^{-1})}{d_\chi}\alpha\overline{\beta}\\
        &=1+ \frac{\overline{\chi(uv^{-1})}}{d_\chi}\overline{\alpha}\beta + \frac{\chi(uv^{-1})}{d_\chi}\alpha\overline{\beta},
    \end{align*}
   or equivalently,  
    \begin{equation*}
        \frac{\overline{\chi(uv^{-1})}}{d_\chi}\overline{\alpha}\beta + \frac{\chi(uv^{-1})}{d_\chi}\alpha\overline{\beta} = 0.
    \end{equation*}
Combining this with the fact that $\alpha,\beta\neq 0$ and  $\overline{\alpha}\beta + \alpha\overline{\beta} = 0$ (by Lemma \ref{lem::1}), we obtain $\chi(uv^{-1})=\overline{\chi(uv^{-1})}\in \mathbb{R}$. Therefore, we conclude that $\chi(uv^{-1}) = \pm d_\chi$ for all $\chi\in \widehat{G}$. According to Lemma  \ref{lem::2}, this is equivalent to the condition that  $|G|$ is even and $uv^{-1}$ is an element of order two in the center of $G$.  Thus  (a) follows.

Let $\widehat{G}_0=\{\chi\in \widehat{G}\mid \chi(uv^{-1})=d_\chi\}$ and $\widehat{G}_1=\{\chi\in \widehat{G}\mid \chi(uv^{-1})=-d_\chi\}$. By the above argument, we have  $\widehat{G}=\widehat{G}_0\cup \widehat{G}_1$. Then it follows from \eqref{eq::2} that 
$$
    \exp(\mathrm{i}t\lambda_\chi)=
    \begin{cases}
    \alpha+\beta, &\mbox{if $\chi\in \widehat{G}_0$},\\
    \alpha-\beta, &\mbox{if $\chi\in \widehat{G}_1$}.
    \end{cases}
    $$
Furthermore, by Lemma \ref{lem::2}, we obtain  
$$\sum_{\chi\in \widehat{G}_0}d_\chi^2 = \sum_{\chi\in \widehat{G}_1}d_\chi^2=\frac{1}{2}\sum_{\chi\in \widehat{G}}d_\chi^2=\frac{|G|}{2}.$$
This proves (b).

 Now suppose that $\alpha = r_\alpha\exp(\mathrm{i}\theta_\alpha)$ and $\beta = r_\beta\exp(\mathrm{i}\theta_\beta)$, where $\theta_\alpha,\theta_\beta\in \mathbb{R}$. Then from $|\alpha|^2 + |\beta|^2 = 1$ and $\overline{\alpha}\beta + \alpha\overline{\beta} = 0$ we obtain $r_\alpha^2 + r_\beta^2 = 1$ and $\theta_\alpha - \theta_\beta = \frac{\pi}{2} + k\pi$ for some $k\in \mathbb{Z}$. Hence, $\alpha \pm \beta = \exp(\mathrm{i}\theta_\beta)((-1)^k r_{\alpha}\mathrm{i} \pm r_\beta)$.
  Write $r_\beta + (-1)^k r_{\alpha}\mathrm{i} = \exp(\mathrm{i}\theta_0)$ with $\theta_0\in \mathbb{R}$. Then we see that 
    \begin{equation*}
        \alpha + \beta = \exp(\mathrm{i}(\theta_0 + \theta_\beta)) ~\mbox{and}~\alpha - \beta = \exp(\mathrm{i}(\pi - \theta_0 + \theta_\beta)).
    \end{equation*}
    Write $t = 2\pi T$, $\theta_0 = 2\pi \mu_0$ and $\theta_\beta = 2\pi \mu_\beta$. According to (b), we immediately deduce that 
    \begin{align}\label{eq::3}
        \begin{cases}
            T\lambda_\chi - \mu_0 - \mu_\beta \in \mathbb{Z},& \mbox{if $\chi\in \widehat{G}_0$,}\\
            T\lambda_\chi + \mu_0 - \mu_\beta -\frac{1}{2}  \in \mathbb{Z},& \mbox{if  $\chi\in \widehat{G}_1$.}
        \end{cases}
    \end{align}
  Since $\sum_{\chi\in\widehat{G}}d_\chi^2\lambda_\chi = \mathrm{tr}(A)=0$ and $\sum_{\chi\in \widehat{G}_0}d_\chi^2 = \sum_{\chi\in \widehat{G}_1}d_\chi^2=\frac{1}{2}\sum_{\chi\in \widehat{G}}d_\chi^2=\frac{|G|}{2}$, it follows from \eqref{eq::3} that
    \begin{align*}
      &~~~~\sum_{\chi\in\widehat{G}_0}d_{\chi}^2\left(T\lambda_\chi - \mu_0 - \mu_\beta\right) +\sum_{\chi\in\widehat{G}_1}d_{\chi}^2\left(T\lambda_\chi +\mu_0 - \mu_\beta-\frac{1}{2}\right) \\
      &=T\sum_{\chi\in\widehat{G}}d_\chi^2\lambda_\chi -\mu_0\left(\sum_{\chi\in\widehat{G}_0}d_\chi^2-\sum_{\chi\in\widehat{G}_1}d_\chi^2\right)-\mu_\beta \sum_{\chi\in\widehat{G}}d_\chi^2-\frac{1}{2}\sum_{\chi\in\widehat{G}_1}d_\chi^2\\
      &=- \mu_\beta|G| -\frac{|G|}{4} \in \mathbb{Z}.
      \end{align*}
Hence, $\mu_\beta \in \mathbb{Q}$. Also note that
    \begin{align*}
        \sum_{\chi\in \widehat{G}_0}d_\chi^2\lambda_\chi &=\sum_{\chi\in \widehat{G}_0}d_\chi^2\cdot \frac{1}{d_\chi}\sum_{s\in S}\chi(s) = \sum_{\chi\in \widehat{G}_0}\sum_{s\in S}d_\chi\chi(s) = \sum_{s\in S}\sum_{\chi\in \widehat{G}_0}d_\chi\chi(s)\\
        &= \sum_{s\in S}\sum_{\chi\in\widehat{G}}\frac{d_\chi + \chi(uv^{-1})}{2}\chi(s)\\
        &= \frac{1}{2}\left(\sum_{s\in S}\sum_{\chi\in\widehat{G}}\chi(s)\overline{\chi(1)} + \sum_{s\in S}\sum_{\chi\in\widehat{G}}\chi(s)\overline{\chi(uv^{-1})}\right)\\
        &= \frac{1}{2} \sum_{s\in S}\sum_{\chi\in\widehat{G}}\chi(s)\overline{\chi(uv^{-1})}\\
        &= \begin{cases}
            \frac{|G|}{2}, &\mbox{if $C_{uv^{-1}}=\{uv^{-1}\}\subseteq S$,}\\
            0, &\mbox{otherwise.}\\
        \end{cases}
    \end{align*}
 As $|G|$ is even, we have $\sum_{\chi\in \widehat{G}_0}d_\chi^2\lambda_\chi\in \mathbb{Z}$,  and so $\sum_{\chi\in \widehat{G}_1}d_\chi^2\lambda_\chi\in  \mathbb{Z}$ due to $\sum_{\chi\in\widehat{G}_0}d_\chi^2\lambda_\chi+\sum_{\chi\in\widehat{G}_1}d_\chi^2\lambda_\chi=\sum_{\chi\in\widehat{G}}d_\chi^2\lambda_\chi =0$.  Let $\chi_0\in \widehat{G}$ denote the character corresponding to the trivial representation of $G$. Clearly, $\chi_0\in\widehat{G_0}$ and $\lambda_{\chi_0}=|S|$. According to \eqref{eq::3}, we have 
\begin{equation}\label{eq::4}
        T|S| - \mu_0 - \mu_\beta \in \mathbb{Z}
    \end{equation}
 and
  \begin{align*}
      &~~~~\sum_{\chi\in\widehat{G}_0}d_{\chi}^2\left(T|S| - \mu_0 - \mu_\beta\right) +\sum_{\chi\in\widehat{G}_1}d_{\chi}^2\left(T\lambda_\chi +\mu_0 - \mu_\beta-\frac{1}{2}\right) \\
      &=T\left(|S|\cdot\sum_{\chi\in\widehat{G}_0}d_\chi^2+\sum_{\chi\in \widehat{G}_1}d_\chi^2\lambda_\chi\right) -\mu_0\left(\sum_{\chi\in\widehat{G}_0}d_\chi^2-\sum_{\chi\in\widehat{G}_1}d_\chi^2\right)-\mu_\beta \sum_{\chi\in\widehat{G}}d_\chi^2-\frac{1}{2}\sum_{\chi\in\widehat{G}_1}d_\chi^2\\
      &=T\left(|S|\cdot \frac{|G|}{2} + \sum_{\chi\in \widehat{G}_1}d_\chi^2\lambda_\chi\right)- \mu_\beta|G| -\frac{|G|}{4} \in \mathbb{Z}.
   \end{align*}
This implies that  $T\in \mathbb{Q}$  due to $\sum_{\chi\in \widehat{G}_1}d_\chi^2\lambda_\chi\in  \mathbb{Z}$ and $\mu_\beta\in \mathbb{Q}$. Then it follows from \eqref{eq::4} that $\mu_0\in \mathbb{Q}$. Again by \eqref{eq::3}, we see that  all the eigenvalues $\lambda_\chi$ with $\chi \in \widehat{G}$ are rational numbers, and so must be integers because they are  algebraic integers. Therefore, $\Gamma$ is an integral graph.

Conversely, suppose that $u,v\in V(\Gamma)=G$ are two vertices  satisfying  the conditions in  (a) and (b). Note that  $\sum_{\chi\in \widehat{G}_0}d_\chi^2=\sum_{\chi\in \widehat{G}_1}d_\chi^2=\frac{|G|}{2}$ by Lemma \ref{lem::2}. Then from \eqref{eq::1} we obtain
    \begin{align*}
        H(t)_{u,u} &= \frac{1}{|G|} \sum_{\chi\in\widehat{G}} d_\chi \chi(uu^{-1})\exp(\mathrm{i}t\lambda_\chi)\\
        &= \frac{1}{|G|} \sum_{\chi\in\widehat{G}}d_\chi^2\exp(\mathrm{i}t\lambda_\chi)\\
        &= \frac{1}{|G|} \left(\sum_{\chi\in \widehat{G}_0}d_\chi^2(\alpha+\beta) + \sum_{\chi\in \widehat{G}_1}d_\chi^2(\alpha-\beta)\right)\\
        &= \frac{1}{|G|} \left(\alpha\sum_{\chi\in\widehat{G}}d_\chi^2 + \beta\left(\sum_{\chi\in \widehat{G}_0}d_\chi^2 - \sum_{\chi\in \widehat{G}_1}d_\chi^2\right)\right)\\
        &= \alpha
    \end{align*}
    and 
    \begin{align*}
        H(t)_{u,v} &= \frac{1}{|G|} \sum_{\chi\in\widehat{G}} d_\chi \chi(uv^{-1})\exp(\mathrm{i}t\lambda_\chi)\\
        &= \frac{1}{|G|} \left(\sum_{\chi\in \widehat{G}_0}d_\chi^2(\alpha+\beta) - \sum_{\chi\in \widehat{G}_1}d_\chi^2(\alpha-\beta)\right)\\
                &= \frac{1}{|G|} \left(\alpha\left(\sum_{\chi\in \widehat{G}_0}d_\chi^2-\sum_{\chi\in \widehat{G}_1}d_\chi^2\right)+\beta\left(\sum_{\chi\in \widehat{G}_0}d_\chi^2+\sum_{\chi\in \widehat{G}_1}d_\chi^2\right)\right)\\
                &= \beta.
    \end{align*}
  Moreover, for any $w\in G\setminus\{u,v\}$, we have
    \begin{align*}
        H(t)_{u,w} &= \frac{1}{|G|} \sum_{\chi\in\widehat{G}} d_\chi \chi(uw^{-1})\exp(\mathrm{i}t\lambda_\chi)\\
        &= \frac{1}{|G|}\left(\sum_{\chi\in \widehat{G}_0}d_\chi\chi(uw^{-1})(\alpha+\beta) + \sum_{\chi\in \widehat{G}_1}d_\chi\chi(uw^{-1})(\alpha-\beta)\right)\\
        &= \frac{1}{|G|} \left(\alpha\sum_{\chi\in\widehat{G}}\chi(uw^{-1})\overline{\chi(1)} + \beta\left(\sum_{\chi\in \widehat{G}_0}d_\chi\chi(uw^{-1}) - \sum_{\chi\in \widehat{G}_1}d_\chi\chi(uw^{-1})\right)\right)\\
        &= \frac{\beta}{|G|}  \left(\sum_{\chi\in \widehat{G}_0}\chi(uw^{-1})\chi(uv^{-1}) + \sum_{\chi\in \widehat{G}_1}\chi(uw^{-1})\chi(uv^{-1})\right)\\
        &= \frac{\beta}{|G|} \sum_{\chi\in\widehat{G}} \chi(uw^{-1})\overline{\chi(uv^{-1})}\\
        &= 0,
    \end{align*}
   where the last equality follows from the fact that $uw^{-1}\neq uv^{-1}$ and $|C_{uv^{-1}}|=1$. As $H(t)$ is unitary, we also have $|\alpha|^2+|\beta|^2=1$. 
    Therefore, we conclude that $(\alpha,\beta)$-revival occurs on $\Gamma$ from  $u$ to  $v$ at time $t$.
    
    We complete the proof.
\end{proof}

According to Theorem \ref{thm::1}, if $G$ is a group of odd order or containing no elements of order two that lies in the center, then every quasi-abelian Cayley graph over $G$ admits no fractional revival. Thus we have the following result.

\begin{corollary}\label{cor::0}
    Let $S_n$ ($n\geq 2$) be the symmetric group  on $[n]$, and let $\Gamma = \mathrm{Cay}(S_n,S)$ be a connected quasi-abelian Cayley graph over $G$. Then $\Gamma$ admits fractional revival if and only if $n=2$.
\end{corollary}

\begin{proof}
If $n=2$, then $S_2=\{(1),(12)\}$ and $\Gamma = \mathrm{Cay}(S_2,\{(12)\})\cong K_2$. It is easy to see that $(\cos(t),\mathrm{i}\sin(t))$-revival occurs on $\Gamma$ from $(1)$ to $(12)$ at  time $t\notin \{k\pi\mid k\in \mathbb{Z}\}$. Now suppose $n\geq 3$, and assume that $(\alpha,\beta)$-revival occurs on $\Gamma$ from a vertex $u$ to a vertex $v$ at time $t$. By Theorem \ref{thm::1}, we see that $uv^{-1}$ is an element of order $2$ that lies in the center of $S_n$. However, this is impossible because the identity $(1)$ is the unique element in the center of  $S_n$ due to $n\geq 3$. Thus the result follows.
\end{proof}

Now suppose that $G$ is an abelian group. For convenience, we express $G$ as
	\begin{equation*}
		G = \mathbb{Z}_{n_1}\oplus\mathbb{Z}_{n_2}\oplus\cdots\oplus\mathbb{Z}_{n_r},
	\end{equation*}
	where  $n_i$ is a power of some prime number for $1 \leq i \leq r$. Then every element $g \in G$ can be uniquely represented as a tuple $g = (g_1, g_2, \ldots, g_r)$, with each $g_i \in \mathbb{Z}_{n_i}$ for $1 \leq i \leq r$. For an element $g \in G$, let  $\chi_g$ be the function from $G$ to  $\mathbb{C}$ defined by letting
	\begin{equation*}
		\chi_g(x) = \prod_{i=1}^r \zeta_{n_i}^{g_i x_i},~\text{for all } x=(x_1,x_2,\ldots,x_r) \in G,
	\end{equation*}
	where $\zeta_{n_i}$ denotes a primitive $n_i$-th root of unity. Then one can verify that $\chi_g$ is  a group homomorphism from $G$ to $\mathbb{C}^*$, and so is  a representation of $G$ with degree $1$. Furthermore, it is well known that $\chi_g$, $g \in G$, form a complete set of  irreducible representations (characters) of $G$, and  $\widehat{G} = \{\chi_g \mid g \in G\}$. Also, by Lemma \ref{lem::2},  we have $\chi_g(x)=\pm 1$ for all $g\in G$ if and only if $x^2=1$. In the case that $x$ is of order $2$, we define 
\begin{equation}\label{eq::5}
G_0=\{g\in G: \chi_g(x)=1\}~\mbox{and}~G_1=\{g\in G: \chi_g(x)=-1\}.
\end{equation} 
Again by Lemma \ref{lem::2},  $|G_0|=|G_1|=|G|/2$. 
Therefore, by Theorem \ref{thm::1}, we immediately deduce the following result due to Cao and Luo \cite{CL22}.

\begin{corollary}(Cao and Luo \cite{CL22}\label{cor::1})
    Let $G$ be an abelian group of even order, and let  $\Gamma={\rm Cay}(G,S)$ be a connected Cayley graph over $G$. Suppose that $\alpha$ and $\beta$ are two nonzero complex numbers. Then  $(\alpha,\beta)$-revival occurs on $\Gamma$ from a vertex $x$ to a vertex $y$ at some time $t$ if and only if the following conditions hold:
    
    \begin{enumerate}[(1)]
   \item $x-y$ is of order two;

\item $\Gamma$ is an integral graph;

\item for every $g\in G_0$, $\exp(\mathrm{i} t \lambda_{\chi_g})=\alpha+\beta$, and for every $g\in G_1$, $\exp(\mathrm{i} t \lambda_{\chi_g})=\alpha-\beta$,
where $G_0$ and $G_1$ are defined by \eqref{eq::5}. 
    \end{enumerate}
\end{corollary}

Let $G$ be a finite group, and let $\Gamma = \mathrm{Cay}(G,S)$ be a connected integral quasi-abelian  Cayley graph over $G$. Then $\lambda_\chi\in \mathbb{Z}$ for all $\chi\in \widehat{G}$. If $|G|=2$, we have $|S|=1$ and $\Gamma\cong K_2$. As in Corollary \ref{cor::0}, we see that $\Gamma$ has fractional revival between its two vertices at any time $t\notin\{k\pi\mid k\in\mathbb{Z}\}$. Now assume that $|G|\geq 3$. Let $a$ be a fixed element of order two that lies in the center of $G$. According to Lemma \ref{lem::2}, we have $\chi(a)=\pm d_\chi$ for all $\chi\in\widehat{G}$. With respect to $a$, we define 
$\widehat{G}_0 = \{\chi\in \widehat{G}\mid \chi(a)=d_\chi\}$ and $\widehat{G}_1=\{\chi\in \widehat{G}\mid \chi(a)=-d_\chi\}$.
Let $\chi'$ be a fixed character in $\widehat{G}_1$. 
We claim that either $\lambda_\chi\neq |S|$ for some $\chi\in\widehat{G}_0$ or   $\lambda_\chi\neq \lambda_{\chi'}$ for some $\chi\in\widehat{G}_1$. If not, then $\Gamma$ has exactly two distinct eigenvalues, namely $|S|$ and $\lambda_{\chi'}$, with the same multiplicity $\frac{|G|}{2}$. Since $\Gamma$ is connected, we must have $|G|=2$
and $|S|=1$, contrary to the assumption. Let $N_a=\{|S|-\lambda_\chi:\chi\in \widehat{G}_0\}\cup \{\lambda_{\chi'}-\lambda_\chi:\chi\in \widehat{G}_1\}$. As $N_a$ contains at least one nonzero integer, we can define 
\begin{equation}\label{eq::6}
   M_a=\mathrm{gcd}(n_a:n_a\in N_a).
\end{equation}

Suppose further that  $(\alpha,\beta)$-revival occurs on $\Gamma$ from a vertex $u$ to a vertex $v$ at time $t$.  Writing $t=2\pi T$, it follows from  \eqref{eq::3} that
\begin{equation}\label{eq::7}
    T(|S|-\lambda_\chi)\in \mathbb{Z} ~\mbox{for all}~ \chi\in \widehat{G}_0~\mbox{and}~ T(\lambda_{\chi'}-\lambda_\chi)\in \mathbb{Z}~\mbox{for all}~  \chi\in \widehat{G}_1. 
\end{equation}
Since $uv^{-1}$ is an element of order two that lies in the center of $G$, we may set  $a=uv^{-1}$ and   $M=M_{uv^{-1}}$. Combining  \eqref{eq::6} with \eqref{eq::7}, we derive
\begin{equation}\label{eq::8}
    TM\in \mathbb{Z}.
\end{equation}
Furthermore, we observe that $M$ cannot equal $1$ and must be a divisor of $|G|$.

\begin{corollary}\label{cor::2}
  Let $G$ be a group with $|G|\geq 3$, and let $\Gamma=\mathrm{Cay}(G,S)$ be a connected 
  quasi-abelian Cayley graph over $G$. If $(\alpha,\beta)$-revival occurs on $\Gamma$ from a vertex $u$ to a vertex $v$ at time $t$, then $M\neq 1$ and $M$ is a divisor of $|G|$, where $M=M_{uv^{-1}}$ is defined by \eqref{eq::6}. In particular, $\mathrm{exp}(\mathrm{i}t)$ is a $|G|$-th root of unity.
\end{corollary}

\begin{proof}
If $M=1$, then \eqref{eq::8} implies that $T\in \mathbb{Z}$. Since $\Gamma$ is an integral graph, we have 
\begin{equation*}
        H(t)=\sum_{\chi\in\widehat{G}}\exp(\mathrm{i}t\lambda_\chi)E_\chi=\sum_{\chi\in\widehat{G}}E_\chi=I,
    \end{equation*}
   which is impossible because $\beta\neq 0$. Thus, $M\neq 1$.

Write $|S| - \lambda_\chi = Mt_\chi$ for $\chi \in \widehat{G}_0$ and $\lambda_{\chi'} - \lambda_\chi = Ms_\chi$ for $\chi \in \widehat{G}_1$. Clearly, $t_\chi, s_\chi \in \mathbb{Z}$ for all $\chi \in \widehat{G}$. Since $|G|\geq 3$, we have $|S| \geq 2$, and so there exists some $w \in S$ such that $w \notin C_{uv^{-1}}=\{uv^{-1}\}$. By the second orthogonality relation for irreducible characters of $G$, we have
\begin{equation}\label{eq::9}
    \sum_{\chi \in \widehat{G}} \lambda_\chi d_\chi \overline{\chi(w)} = \sum_{\chi \in \widehat{G}} \left( \sum_{s \in S} \frac{\chi(s)}{d_\chi} \right) d_\chi \overline{\chi(w)} = \sum_{s \in S} \sum_{\chi \in \widehat{G}} \chi(s) \overline{\chi(w)} = \frac{|G|}{|C_w|},
\end{equation}
where $C_w$ is the conjugacy class of $G$ containing $w$. On the other hand,
\begin{equation}\label{eq::10}
    \sum_{\chi \in \widehat{G}} \lambda_\chi d_\chi \overline{\chi(w)} = \sum_{\chi \in \widehat{G}_0} (|S| - Mt_\chi) d_\chi \overline{\chi(w)} + \sum_{\chi \in \widehat{G}_1} (\lambda_{\chi'} - Ms_\chi) d_\chi \overline{\chi(w)}.
\end{equation}
Moreover, since $w\neq 1$ and $w \notin C_{uv^{-1}}$, we have
\begin{equation}\label{eq::11}
\left\{
\begin{aligned}
    &\sum_{\chi \in \widehat{G}_0} |S| d_\chi \overline{\chi(w)} = |S| \sum_{\chi \in \widehat{G}} \frac{d_\chi + \chi(uv^{-1})}{2} \overline{\chi(w)} = 0, \\
    &\sum_{\chi \in \widehat{G}_1} \lambda_{\chi'} d_\chi \overline{\chi(w)} = \lambda_{\chi'} \sum_{\chi \in \widehat{G}} \frac{d_\chi - \chi(uv^{-1})}{2} \overline{\chi(w)} = 0.
\end{aligned}
\right.
\end{equation}
Combining \eqref{eq::9}, \eqref{eq::10} and \eqref{eq::11}, we obtain
\begin{equation}\label{eq::12}
    \frac{|G|}{M} = -|C_w| \sum_{\chi \in \widehat{G}_0} t_\chi d_\chi \overline{\chi(w)} - |C_w| \sum_{\chi \in \widehat{G}_1} s_\chi d_\chi \overline{\chi(w)}.
\end{equation}
Note that $t_\chi, s_\chi \in \mathbb{Z}$ and $\overline{\chi(w)}$ are algebraic integers for all $\chi \in \widehat{G}$. Then from \eqref{eq::12} we deduce that  $\frac{|G|}{M}$ is an algebraic integer, and so must be an integer. Therefore, $M$ is a divisor of $|G|$, and so $\mathrm{exp}(\mathrm{i}t)=\mathrm{exp}(2\pi \mathrm{i}T)$ is a $|G|$-th root of unity by \eqref{eq::8}.
 The result follows.
\end{proof}

At the end of this paper, we provide an example of quasi-abelian Cayley graph that admits fractional revival.

\begin{example}
\rm
Let $G = \mathbb{Z}_6 \times D_3$, where $D_3 = \langle a, b \mid a^3 = b^2 = e, \, b^{-1}ab = a^{-1} \rangle$. Consider the set $S = \{(1,a), (5,a), (1,a^2), (5,a^2), (1,b), (5,b), (1,ba), (5,ba), (1,ba^2), (5,ba^2)\}$. It is straightforward to verify that the Cayley graph $\Gamma = \text{Cay}(G, S)$ is connected and quasi-abelian. Let $u = (0, e)$ and $v = (3, e)$. Then $uv^{-1} = (3, e)$, which is an element of order $2$ that lies in the center of $G$. By Lemma \ref{lem::2}, we have $\chi(uv^{-1}) = \pm d_{\chi}$ for all $\chi \in \widehat{G}$.  For any $\chi \in \widehat{G}$, there exist unique characters $\varphi \in \widehat{\mathbb{Z}_6}$ and $\psi \in \widehat{D_3}$ such that $\chi = \varphi \otimes \psi$, i.e., $\chi(g) = \varphi(k)\psi(h)$ for all $g = (k, h) \in G$. Table \ref{tab::1} lists the irreducible characters of $\mathbb{Z}_6$ and $D_3$. With respect to $uv^{-1}$, the characters of $G$ partition into:
 \begin{table}[t]
    \centering
    \caption{Character tables of $\mathbb{Z}_6$ and $D_3$.}
    \label{tab::1}
    \begin{tabular}{*{7}{c}|*{4}{c}}
        \toprule
        \multicolumn{7}{c|}{$\mathbb{Z}_6$} & \multicolumn{4}{c}{$D_3$} \\
        \cmidrule(r){1-7} \cmidrule(l){8-11}
         Class&$\{0\}$&$\{1\}$&$\{2\}$&$\{3\}$&$\{4\}$&$\{5\}$&Class&$\{e\}$&$\{a,a^2\}$&$\{b,ba,ba^2\}$\\
                \midrule
                $\varphi_1$&1&1&1&1&1&1&$\psi_1$&1&1&1\\
                $\varphi_2$&1&$\exp(\frac{\pi\mathrm{i}}{3})$&$\exp(\frac{2\pi\mathrm{i}}{3})$&$-1$&$\exp(\frac{4\pi\mathrm{i}}{3})$&$\exp(\frac{5\pi\mathrm{i}}{3})$&$\psi_2$&1&1&$-1$\\
                $\varphi_3$&1&$\exp(\frac{2\pi\mathrm{i}}{3})$&$\exp(\frac{4\pi\mathrm{i}}{3})$&$1$&$\exp(\frac{2\pi\mathrm{i}}{3})$&$\exp(\frac{4\pi\mathrm{i}}{3})$&$\psi_3$&2&$-1$&0\\
                $\varphi_4$&1&$-1$&1&$-1$&1&$-1$\\
                $\varphi_5$&1&$\exp(\frac{4\pi\mathrm{i}}{3})$&$\exp(\frac{2\pi\mathrm{i}}{3})$&1&$\exp(\frac{4\pi\mathrm{i}}{3})$&$\exp(\frac{2\pi\mathrm{i}}{3})$\\
                $\varphi_6$&1&$\exp(\frac{5\pi\mathrm{i}}{3})$&$\exp(\frac{4\pi\mathrm{i}}{3})$&$-1$&$\exp(\frac{2\pi\mathrm{i}}{3})$&$\exp(\frac{\pi\mathrm{i}}{3})$\\
    \bottomrule
    \end{tabular}
\end{table}
\begin{equation*}
\begin{aligned}
    \widehat{G}_0 &= \{\varphi_1 \otimes \psi_1, \varphi_1 \otimes \psi_2, \varphi_1 \otimes \psi_3, \varphi_3 \otimes \psi_1, \varphi_3 \otimes \psi_2, \varphi_3 \otimes \psi_3, \varphi_5 \otimes \psi_1, \varphi_5 \otimes \psi_2, \varphi_5 \otimes \psi_3\}, \\
    \widehat{G}_1 &= \{\varphi_2 \otimes \psi_1, \varphi_2 \otimes \psi_2, \varphi_2 \otimes \psi_3, \varphi_4 \otimes \psi_1, \varphi_4 \otimes \psi_2, \varphi_4 \otimes \psi_3, \varphi_6 \otimes \psi_1, \varphi_6 \otimes \psi_2, \varphi_6 \otimes \psi_3\}.
\end{aligned}
\end{equation*}
Then a direct calculation shows that
\begin{align*}
    \exp\left(\frac{2\pi\lambda_\chi\mathrm{i}}{3}\right) &= 
    \begin{cases}
        \exp\left(\frac{2\pi\mathrm{i}}{3}\right) = \cos\frac{2\pi}{3} + \mathrm{i}\sin\frac{2\pi}{3}, & \text{if } \chi \in \widehat{G}_0, \\
        \exp\left(\frac{4\pi\mathrm{i}}{3}\right) = \cos\frac{2\pi}{3} - \mathrm{i}\sin\frac{2\pi}{3}, & \text{if } \chi \in \widehat{G}_1.
    \end{cases}
\end{align*}
By Theorem \ref{thm::1}, we conclude that $(\cos\frac{2\pi}{3}, \mathrm{i}\sin\frac{2\pi}{3})$-revival occurs on $\Gamma$ from vertex $u = (0, e)$ to vertex $v = (3, e)$ at time $t = \frac{2\pi}{3}$.
\end{example}

	\section*{Declaration of Interest Statement}
	
	The authors declare that they have no known competing financial interests or personal relationships that could have appeared to influence the work reported in this paper.
	
	\section*{Acknowledgements}
	
	X. Huang was supported by the National Natural Science Foundation of China (No. 12471324) and the Natural Science Foundation of Shanghai (No. 24ZR1415500).  X. Liu was supported by the National Natural Science Foundation of China (No. 12371358) and the Guangdong Basic and Applied Basic Research Foundation (No.
2023A1515010986).
	
	\section*{Data availability}
	No data was used for the research described in the article.

\end{document}